\documentclass[11pt]{article}
\usepackage{amsmath,amssymb,amsthm}
\usepackage{color}
\usepackage[english]{babel}
\usepackage[utf8]{inputenc}

\usepackage{fancybox}
\usepackage{listings}
\usepackage{mathtools}
\usepackage{verbatim}

\usepackage{float}
\usepackage{makeidx}

\newtheorem{Theorem}{Theorem}[section]

\newtheorem{Definition}[Theorem]{Definition}

\newtheorem{Proposition}[Theorem]{Proposition}

\newtheorem{Lemma}[Theorem]{Lemma}

\newtheorem{Corollary}[Theorem]{Corollary}

\newtheorem{Remark}[Theorem]{Remark}

\newcommand{\RR}{{{\rm I} \kern -.15em {\rm R} }}

\newcommand{\C}{{{\rm l} \kern -.42em {\rm C} }}

\newcommand{\nat}{{{\rm I} \kern -.15em {\rm N} }}

\newcommand{\be}{\begin{equation}}
\newcommand{\ee}{\end{equation}}
\newcommand{\beq}{\begin{eqnarray}}
\newcommand{\eeq}{\end{eqnarray}}
\newcommand{\beqs}{\begin{eqnarray*}}
\newcommand{\eeqs}{\end{eqnarray*}}
\newcommand{\bt}{\begin{Theorem}}
\newcommand{\et}{\end{Theorem}}
\newcommand{\br}{\begin{Remark}}
\newcommand{\er}{\end{Remark}}
\newcommand{\bc}{\begin{Corollary}}
\newcommand{\ec}{\end{Corollary}}
\newcommand{\bl}{\begin{Lemma}}
\newcommand{\el}{\end{Lemma}}
\newcommand{\bd}{\begin{definition}}
\newcommand{\ed}{\end{definition}}
\newcommand{\bp}{\begin{Proposition}}
\newcommand{\eP}{\end{Proposition}}

\title{Opinion dynamics of  two populations with time-delayed coupling}

\bigskip

\author{Chiara Cicolani\footnote{Email: chiara.cicolani@graduate.univaq.it}\hspace{0.3cm}\&\hspace{0.2cm}Cristina Pignotti\footnote{Email: cristina.pignotti@univaq.it.} \\Dipartimento di Ingegneria e Scienze dell'Informazione e Matematica\\
		Universit\`{a} degli Studi di L'Aquila\\
		Via Vetoio, Loc. Coppito, 67100 L'Aquila Italy}
	
\begin{document}
\maketitle

\begin{abstract}
We study a Hegselmann-Krause type opinion formation model for a system of two populations. The two groups interact with each other via subsets of individuals, namely the leaders, and natural time delay effects are considered. By using careful estimates of the system's trajectories, we are able to prove an asymptotic convergence to consensus result. Some numerical tests illustrate the theoretical result and point out some possible applications.
  \end{abstract}

\section{Introduction}
Due to several applications in many scientific fields, multiagent systems have become in the last years a very attractive research topic. They naturally appear e.g. in biology \cite{Cama, Carrillo, CS1}, ecology \cite{Sole}, economics \cite{Carrillo, Marsan}, social sciences \cite{Bellomo, BN, BHT, Castellano, Campi}, physics \cite{DH, Stro}, control theory \cite{D, PaolucciP, PRT, WCB}, engineering and robotics \cite{Bullo, Desai}. For more and other applications see also \cite{Hel, Jack, XWX}.
An important feature often analyzed is the possible emergence of self-organization leading the group's agents to globally collective behaviors. Here, we are interested in the celebrated Hegselmann-Krause  model for opinion formation, originally proposed in \cite{HK}. Since then, several generalizations have been proposed (see e.g. \cite{Bellomo, BN, CFT, CFT2, Ceragioli, JM}).

One of the most natural extensions of the Hegselmann-Krause model concerns the analysis of the time-delayed interactions, in order to take into account the times necessary for each agent to receive information from other agents or  reaction times. Opinion formation models in presence of time delay effects have  already been studied  by several authors, see e.g.  \cite{CPP,2,onoff,3,4,P,PaolucciP}. Concerning the second-order version of the Hegselmann-Krause model, namely the Cucker-Smale model \cite{CS1}, introduced to describe flocking phenomena, delayed interactions have also been considered in many papers, see e.g.   \cite{Chen,CH,CL,CP,Cont, EHS,HM,PR,PT,Cartabia}.

 In this paper, we are interested in studying the convergence to consensus for  a Hegselmann-Krause type opinion formation model involving two populations with time-delayed coupling. Indeed, it is natural to assume that, while one can consider almost instantaneous the influence among agents in the same population, a certain time lag appears in the interaction among individuals of different populations.  A more general model would include time delays (eventually smaller) also in the interactions in the same population. Here, for simplicity, we choose to consider delay effects only in the interactions among agents of different populations. An analogous analysis could be performed also in the more  general situation with multiple delays.

Consider two finite sets of $N$ and $M$ agents respectively, with $N, M \in \nat$, $N, M \geq 2 \ .$ Without loss of generality, we assume $M \leq N.$
Let $x_i(t) \in \RR^d, i=1,\dots, N,$ be the opinion of the i-th particle of the first family at time $t$ and $y_i(t) \in \RR^d$, $i=1,...,M,$ be  the opinion of the i-th particle of the second family at time $t.$  We consider that a (small) group of agents of the first family interacts with another (small) group of the second family with a time delay appearing as a time needed from an agent of a population to receive information from agents of the other one. The time delay is assumed to be a positive constant, $\tau>0.$ Given $h,k \in \nat$, 
$h < M \mbox{ and } k < N $, the opinions of the two populations evolve following  the 
Hegselmann-Krause opinion formation model: 
\begin{equation}
\begin{split}
    & \frac{d}{d t}x_i(t)= \sum_{ j \neq i } a_{ij}(t)(x_j(t)-x_i(t)) + \sum_{j=1}^{h} \epsilon_{ij}(t)(y_j(t-\tau)-x_i(t)), \quad t>0,\ i=1,...,k, \\
    & \frac{d}{d t}x_i(t)= \sum_{j \neq i } a_{ij}(t)(x_j(t)-x_i(t)), \quad t>0,\ i=k+1,...,N, \\
    & \frac{d}{d t}y_i(t)= \sum_{j \neq i } b_{ij}(t)(y_j(t)-y_i(t)) + \sum_{j=1}^{k} \eta_{ij}(t)(x_j(t-\tau)-y_i(t)), \quad t>0,\ i=1,...,h, \\
    & \frac{d}{d t}y_i(t)= \sum_{j \neq i } b_{ij}(t)(y_j(t)-y_i(t)), \quad t>0,\ i=h+1,...,M, 
    \end{split}
\label{eq1}
\end{equation}

with the interaction weights $a_{ij}(t), t\ge 0,$ of the form:
\begin{equation}
    \begin{split}
        a_{ij}(t):= \frac{1}{N+h-1}\psi(x_i(t), x_j(t)), \quad \  i=1,...,k, \ j=1, \dots, N,\\
        a_{ij}(t):= \frac{1}{N-1} \psi(x_i(t),x_j(t)), \quad \  i=k+1,...,N,\ j=1, \dots, N,\\
    \end{split}
    \label{eq2}
\end{equation}

and the weights $b_{ij}(t), t\ge 0,$ of the form:
\begin{equation}
    \begin{split}
        b_{ij}(t):= \frac{1}{M+k-1}\psi^*(y_i(t), y_j(t)), \quad  \ i=1,...,h,\ j=1, \dots, M, \\
        b_{ij}(t):= \frac{1}{M-1} \psi^*(y_i(t), y_j(t)), \quad \ i=h+1,...,M, \ j=1, \dots, M.\\
    \end{split}
    \label{eq3}
\end{equation}

Here, $\psi: \RR^d\times\RR^d\rightarrow \RR$ and $\psi^*: \RR^d\times\RR^d\rightarrow \RR$ are continuous, positive and bounded functions. Moreover, the interaction coefficients $\epsilon_{ij}(t)$ and $\eta_{ij}(t),$ for $t\ge 0,$ among individuals of different populations have the form:
\begin{equation}
\begin{split}
    \epsilon_{ij}(t):= \frac{1}{N+h-1}\phi(x_i(t), y_j(t-\tau)), \quad  \ i=1,...,k,\ j=1, \dots, h,\\
    \eta_{ij}(t):= \frac{1}{M+k-1}\phi^*(y_i(t),  x_j(t-\tau)), \quad  \ i=1,...,h,\ j=1, \dots, k,
\end{split} 
\end{equation}
where $\phi: \RR^d\times\RR^d \rightarrow \RR$ and $\phi^*: \RR^d\times\RR^d\rightarrow \RR$ are continuous, positive and bounded functions.
Let  us denote
$$ \Lambda := \max \Big \{ \| \psi\|_{\infty}, \| \psi^* \|_{\infty}, \| \phi \|_{\infty}, \| \phi^* \|_{\infty} \Big \}\,. $$

Let us assume the initial conditions: 
\begin{equation}
    \begin{cases}
        x_i(t)=x_i^0(t), \quad i=1,...,k, \ t \in [-\tau,0], \\
        x_i(0)=x_i^0, \quad i=k+1,...,N, \\
        y_i(t)=y_i^0(t), \quad i=1,...,h, \ t \in [-\tau,0], \\
        y_i(0)=y_i^0, \quad i=h+1,...,M, 
    \end{cases}
    \label{init}
\end{equation}
where $x_i^0(\cdot), i=1, \dots, k,$  $y_i^0(\cdot), i=1, \dots, h,$ are continuous functions defined on $[-\tau, 0],$ $x_i^0\in \RR^d,$ $i=k+1, \dots, N,$   $y_i^0\in\RR^d,$  $i=h+1, \dots, M.$   

For well-posedness results for model \eqref{eq1}-\eqref{init},  we refer to classical texts on functional differential equations \cite{Halanay, Hale}. Here, we will focus on
the asymptotic behavior of the solution. In particular, we want to prove the convergence to \emph{consensus}. Let us define the \emph{diameter} of each population as

$$d_X(t):= \max_{i,j=1,...,N} |x_i(t)-x_j(t)|, \quad
d_Y(t):=\max_{i,j=1,...,M} |y_i(t)-y_j(t)|.$$
Moreover, let us define the \emph{global diameter} as
$$ d(t):= \max \Big \{ d_X(t), d_Y(t), \max_{i=1,...,N} \max_{j=1,...,M} |x_i(t)-y_j(t)|  \Big \}. $$

\begin{Definition} \label{def1} We say that the solution $(x_i(t),y_j(t)), i=1,...,N, j=1,...,M,$ to system \eqref{eq1}, with initial condition \eqref{init}, converges to consensus if $$ \lim_{t \rightarrow \infty} d(t)=0$$
\end{Definition}

This kind of model can have applications in social sciences, economics, politics, and ecology. Indeed, it is reasonable to try reaching a global consensus among individuals of different countries,  or different groups of individuals in the same country,  about important questions such as, e.g., ecological behaviors, climate change's reasons, appropriate strategies to reduce $CO_2$ emissions, etc. The proof of a consensus result for model \eqref{eq1} can be considered as a first insight for more quantitative studies aiming to design appropriate control strategies.

The rest of the paper is organized as follows. In section \ref{Prel} we introduce some notations and give preliminary lemmas. In section \ref{cons} we state and prove the consensus result. Finally, in section \ref{num}, we present some numerical tests validating the theoretical results and discuss some possible applications
to ecology and sustainability.

\section{Preliminaries}\label{Prel}
 In this section, we present some preliminary results useful for studying the consensus behavior. 

Firstly, for any fixed a vector $v \in \mathbb{R}^d,$ let us define the following quantities:
\begin{equation} \label{m0}
\begin{split}
m_0:= \min \Big \{  & \min_{i=1,\dots,k} \min_{t \in [-\tau,0]} \langle x_i(t),v \rangle, \min_{i=k+1,\dots,N} \langle x_i(0),v \rangle, \\
& \hspace{1 cm}\min_{i=1,\dots,h} \min_{t \in [-\tau,0]} \langle y_i(t),v \rangle, \min_{i=h+1,\dots,M} \langle y_i(0),v \rangle \Big \},
\end{split}
\end{equation}
and
\begin{equation} \label{M0}
\begin{split}
M_0:= \max \Big \{ & \max_{i=1,\dots,k} \max_{t \in [-\tau,0]} \langle x_i(t),v \rangle, \max_{i=k+1,\dots,N} \langle x_i(0),v \rangle, \\
& \hspace{1 cm} \max_{i=1,\dots,h} \max_{t \in [-\tau,0]} \langle y_i(t),v \rangle, \max_{i=h+1,\dots,M} \langle y_i(0),v \rangle \Big \}.
\end{split}
\end{equation}

\begin{Lemma} \label{2.1}
Let $(x_i(t),y_j(t)), \ i=1\dots,N, \ j=1,\dots,M,$ be a global classical solution of the system \eqref{eq1}-\eqref{init}. Then, for all $v \in \mathbb{R}^d$ we have that 
\begin{equation} \label{1}
m_0 \leq \langle x_i(t),v \rangle \leq M_0, \end{equation}
and 
\begin{equation} \label{2}
 m_0 \leq \langle y_j(t),v \rangle \leq M_0, 
\end{equation}
for all $t \geq -\tau,$ $i=1,\dots,k$, $j=1,\dots,h,$ and for all $t\geq 0,$ $i=k+1,\dots,N,$ $j=h+1,\dots,M.$
\end{Lemma}

\begin{proof} Fix a vector  $v \in \mathbb{R}^d$  and let $m_0, M_0$ be the constants defined in \eqref{m0}, \eqref{M0}. 
By definition of $m_0$ and $M_0$, we have that the inequalities \eqref{1} and \eqref{2} are trivially satisfied for $t \in [-\tau,0],$  $i=1,\dots,k$ and $j=1,\dots,h.$  Then, we want to prove \eqref{1} and \eqref{2} for $t\ge 0.$
Let us prove \eqref{1}; \eqref{2} follows analougously. 

For a fixed parameter   $\epsilon >0,$ let us define the following set:
$$
\begin{array}{l}
\displaystyle{
 T^{\epsilon}:= \left \{ t>0 \ :\  \langle x_i(s),v \rangle < M_0+\epsilon, \ \forall i=1,\dots, N,\right.}
\\
\displaystyle{\left.
\hspace{2 cm}\langle y_i(s),v \rangle < M_0+\epsilon, \ \forall i=1,\dots, M, \ \forall s \in [0,t) \right \}}.
\end{array} $$
By continuity, $T^{\epsilon} \neq \emptyset.$ Let us call $S^{\epsilon}:= \sup T^{\epsilon}.$ We want to prove that $S^{\epsilon}=+\infty.$ Let us suppose by contraddiction that $S^{\epsilon}<+\infty.$ Then,
\begin{equation}
\max_{i=1,\dots,N} \langle x_i(t), v \rangle < M_0+\epsilon, \ \forall t \in [0,S^{\epsilon}),
\end{equation}
and 
\begin{equation}
\lim_{t \rightarrow S^{\epsilon -}} \max_{i=1,\dots,N} \langle x_i(t), v \rangle = M_0+\epsilon.
\end{equation}
For all $t \in [0,S^{\epsilon})$ and $i=k+1,\dots,N$ we have 
\begin{equation}
\begin{split}
\frac{d}{d t} \langle x_i(t),v \rangle & = \sum_{j \neq i} a_{ij}(t) \langle x_j(t)-x_i(t), v \rangle \\
& \leq \sum_{j \neq i} a_{ij}(t) (M_0+\epsilon-\langle x_i(t),v \rangle) \\
& \leq \Lambda (M_0+\epsilon-\langle x_i(t),v \rangle).
\end{split}
\end{equation}
Applying the Grönwall's Lemma over $t \in [0,S^{\epsilon})$, we find
\begin{equation}
\begin{split}
\langle x_i(t),v \rangle & \leq e^{-\Lambda t} \langle x_i(0),v \rangle +(M_0+\epsilon)(1-e^{-\Lambda t}) \\
& \leq M_0+\epsilon -\epsilon e^{-\Lambda t} \leq M_0+\epsilon -\epsilon e^{-\Lambda S^{\epsilon}},
\end{split}
\end{equation}
for all $t \in [0,S^{\epsilon})$ and $i=k+1,\dots,N.$
Therefore, we deduce that 
$$ \max_{i=k+1,\dots,N} \langle x_i(t),v \rangle < M_0+\epsilon -\epsilon e^{-\Lambda S^{\epsilon}}, \ \forall t \in (0,S^{\epsilon}). $$
Consider now $t \in [0,S^{\epsilon})$ and $i=1,\dots,k.$ 
Then, we have 
\begin{equation}
\begin{array}{l}
\displaystyle{\frac{d}{d t} \langle x_i(t),v \rangle  = \sum_{j \neq i} a_{ij}(t)\langle x_j(t)-x_i(t),v \rangle + \sum_{j=1}^{h} \epsilon_{ij}(t) \langle y_j(t-\tau)-x_i(t),v \rangle} \\
\displaystyle{\hspace{2cm}\leq \sum_{j \neq i} a_{ij}(t) (M_0+\epsilon-\langle x_i(t),v \rangle) + \sum_{j=1}^{h} \epsilon_{ij}(t)(M_0+\epsilon-\langle x_i(t),v \rangle)} \\
 \displaystyle{\hspace{2cm}\leq \Lambda (M_0+\epsilon-\langle x_i(t),v \rangle).}
\end{array}
\end{equation}
Therefore, we can find analogously that 
$$ \max_{i=1,\dots,k} \langle x_i(t),v \rangle < M_0+\epsilon -\epsilon e^{-\Lambda S^{\epsilon}}, \ \forall t \in [0,S^{\epsilon}). $$
Then, we have that
$$ \max_{i=1,\dots,N} \langle x_i(t),v \rangle < M_0+\epsilon -\epsilon e^{-\Lambda S^{\epsilon}}, \ \forall t \in [0,S^{\epsilon}). $$ 
Passing to the limit for $t \rightarrow S^{\epsilon -}$, we find 
$$ \lim_{t \rightarrow S^{\epsilon -}} \max_{i=1,\dots,N} \langle x_i(t),v \rangle \leq  M_0+\epsilon -\epsilon e^{-\Lambda S^{\epsilon}} < M_0 + \epsilon, $$
and this gives a contradiction. Then, we have that $S^{\epsilon}=+\infty.$
Hence, by arbitrarity of $\epsilon$, we have that 
$$ \max_{i=1,\dots,N} \langle x_i(t),v \rangle \leq M_0, \ \forall t \geq 0, \ v \in \mathbb{R}^d.$$
Therefore, we have that 
$$\langle x_i(t),v \rangle \leq M_0, \ \forall \ t \geq -\tau, \ i=1,\dots,k, $$
and 
$$\langle x_i(t),v \rangle \leq M_0, \ \forall t \geq 0, \ i=k+1,\dots,N. $$
In order to prove the other inequality we observe that, by the proven estimate,
\begin{equation}
\begin{split}
& - \langle x_i(t),v \rangle  = \langle x_i(t), -v \rangle\hspace{2 cm} \\
& \hspace{2 cm}\leq \max \Big \{ \max_{i=1,\dots,k} \max_{t \in [-\tau,0]} \langle x_i(t),-v \rangle, \max_{i=k+1,\dots,N} \langle x_i(0),-v \rangle, \\
&\hspace{3 cm} \ \ \ \ \ \ \ \ \ \ \max_{i=1,\dots,h} \max_{t \in [-\tau,0]} \langle y_i(t),-v \rangle, \max_{i=h+1,\dots,M} \langle y_i(0),-v \rangle \Big \} \\
& \hspace{2 cm}= - \min \Big \{ \min_{i=1,\dots,k} \min_{t \in [-\tau,0]} \langle x_i(t),v \rangle, \min_{i=k+1,\dots,N} \langle x_i(0),v \rangle, \\
& \hspace{3 cm}\ \ \ \ \ \ \ \ \ \ \min_{i=1,\dots,h} \min_{t \in [-\tau,0]} \langle y_i(t),v \rangle, \min_{i=h+1,\dots,M} \langle y_i(0),v \rangle \Big \}= -m_0.
\end{split}
\end{equation}
This concludes the proof.
\end{proof}
The above lemma allows us to deduce a bound on the states.
\begin{Lemma} \label{2.4}
Let $(x_i(t),y_j(t)), \ i=1\dots,N, \ j=1,\dots,M,$ be a global classical solution of the system \eqref{eq1}-\eqref{init}. Then, 
\begin{equation}\label{eq16}
|x_i(t)| \leq C_0, \quad |y_j(t)| \leq C_0,
\end{equation}
 $\forall \, t\ge -\tau,$ for   $i=1, \dots, k,$ $j=1, \dots, h,$ and $\forall \, t\ge 0,$ for $i=k+1, \dots, N,$ $j=h+1, \dots, M,$
where $C_0$ is given by
$$ C_0:= \max \Big \{ \max_{i=1,\dots,k} \max_{t \in [-\tau,0]} |x_i(t)|, \max_{i=k+1,\dots,N} |x_i(0)|, \max_{j=1,\dots,h} \max_{t \in [-\tau,0]} |y_j(s)|, \max_{j=h+1,\dots,M} |y_j(0)| \Big \}. $$
\end{Lemma}

\begin{proof} We prove the first inequality of \eqref{eq16}. The second one for $|y_j(t)|, \ j=1,\dots,M,$ follows analogously. For $i=1,\dots,N$ and $t \geq -\tau,$ if $|x_i(t)|=0,$ the result is trivial. Let us suppose $|x_i(t)|>0$ and define the vector
    $$ v:=\frac{x_i(t)}{|x_i(t)|} \ .$$
Then, applying \eqref{1} and the Cauchy-Schwarz inequality, we get
\begin{equation} \begin{split}
& |x_i(t)|= \langle x_i(t),v \rangle \leq M_0 \\
& \hspace{1,2 cm}= \max \Big \{ \max_{i=1,\dots,k} \max_{t \in [-\tau,0]} \langle x_i(t),v \rangle, \max_{i=k+1,\dots,N} \langle x_i(0),v \rangle, \max_{i=1,\dots,h} \max_{t \in [-\tau,0]} \langle y_i(t),v \rangle, \\
&  \hspace{3,2 cm}\hspace{3 cm} \max_{i=h+1,\dots,M} \langle y_i(0),v \rangle \Big \} \\
&  \hspace{1,2 cm}\leq  \max \Big \{ \max_{i=1,\dots,k} \max_{t \in [-\tau,0]} |x_i(t)| |v|, \max_{i=k+1,\dots,N} |x_i(0)||v|, \max_{j=1,\dots,h} \max_{t \in [-\tau,0]} |y_j(s)||v|, \\
& \hspace{3,2 cm} \hspace{3 cm} \max_{j=h+1,\dots,M} |y_j(0)||v| \Big \} =C_0,
\end{split}
\end{equation}
being v a unit vector.  So, the first inequality of \eqref{eq16} is proven.
\end{proof}
\begin{Remark} \label{2.5}
    From Lemma \ref{2.4}, since the influence functions $\psi$ and $\psi^*$ are continuous, we deduce that
    \begin{equation}
        \begin{split}
            \psi(x_i(t),x_j(t)) \geq \psi_0:= \min_{|z_1|,|z_2| \leq C_0} \psi(z_1,z_2) >0  , \\
            \psi^*(y_l(t),y_r(t)) \geq \psi^*_0:= \min_{|z_1|,|z_2| \leq C_0} \psi^*(z_1,z_2) >0  ,
        \end{split}
    \end{equation}
    for each $t \geq 0 $, $i,j=1,...,N$ and $l,r=1,...,M$. Moreover, since the functions $\phi$ and $\phi^*$ are continuous too, again we deduce that
    \begin{equation}
        \begin{split}
            \phi(x_i(t),y_j(t-\tau)) \geq \phi_0:= \min_{|z_1|,|z_2| \leq C_0} \phi(z_1,z_2) >0 , \\
            \phi^*(y_l(t),x_r(t-\tau)) \geq \phi^*_0:= \min_{|z_1|,|z_2| \leq C_0} \phi^*(z_1,z_2) >0 ,
        \end{split}
    \end{equation}
    for each $t \geq 0 $, $i,r=1,...,k$ and $j,l=1,...,h$.
\end{Remark}

From Remark \ref{2.5}, we can define the positive constant
\begin{equation}\label{Gamma} 
 \Gamma := \min \Big \{ \psi_0, \psi_0^*, \phi_0, \phi_0^* \Big \}\,.
\end{equation}

Now, fix $v \in \mathbb{R}^d$ and let $m_0,$ $M_0$ be as in \eqref{m0} and \eqref{M0} respectively. Since, up to changes of influence function,  system \eqref{eq1} is invariant by translation, without loss of generality, we may assume

$$0<m_0\le M_0.$$

Inspired by \cite{3}, we can prove the following lemma.
\begin{Lemma} \label{case1}
Let $(x_i(t),y_j(t)),$ with $i=1,\dots,N$ and $j=1,\dots,M,$ be a global classical solution of the system \eqref{eq1}-\eqref{init}. Then, for $t \in [5\tau,6\tau],$ we have 
\begin{equation} \label{x}
 m_0+\frac{\Gamma_1}{2}(M_0-m_0) \leq \langle x_i(t),v \rangle \leq M_0-\frac{\Gamma_1}{2}(M_0-m_0), 
 \end{equation}
 and 
 \begin{equation}\label{y}
 m_0+\frac{\Gamma_1}{2}(M_0-m_0) \leq \langle y_j(t),v \rangle \leq M_0-\frac{\Gamma_1}{2}(M_0-m_0),
 \end{equation}
for a suitable constant $\Gamma_1\in (0,1).$
\end{Lemma}

\begin{proof}  Let us proceed by steps.

\emph{Step 1:} Suppose that exists $L \in \{k+1,\dots,N\}$ such that $\langle x_L(0),v \rangle = m_0.$ Then, since $|\langle \dot{x}_i(t),v \rangle|\leq 2 \Lambda M_0, \ \forall t \geq 0,$ we have that 
$$ m_0 \leq \langle x_L(t), v \rangle \leq \frac{M_0+m_0}{2}, \ t \in [0,\sigma],$$
where $\sigma$ is a positive number such that
\begin{equation}\label{sigma}
\sigma \le\min \Big \{ \tau, \frac{M_0-m_0}{4 \Lambda M_0} \Big \}. 
\end{equation}
Consider $i \in \{k+1,\dots,N \} \setminus \{L\}$ and $t \in [0,\sigma].$ Then,
\begin{equation}
\begin{split}
\frac{d}{d t} \langle x_i(t),v \rangle & = \sum_{\substack{j \neq i \\ j \neq L}} a_{ij}(t) \langle x_j(t)-x_i(t),v \rangle + a_{iL}(t) \langle x_L(t)-x_i(t),v \rangle \\
& \leq \sum_{\substack{j \neq i \\ j \neq L}} a_{ij}(t) (M_0 - \langle x_i(t),v \rangle )+ a_{iL}(t)\Big(\frac{M_0+m_0}{2}- \langle x_i(t),v \rangle\Big) \\
& = \Big( \sum_{j \neq i} a_{ij}(t)-a_{iL}(t) \Big) (M_0 - \langle x_i(t),v \rangle )+ a_{iL}(t)\Big(\frac{M_0+m_0}{2}- \langle x_i(t),v \rangle\Big) \\
& \leq (\Lambda-a_{iL}(t)) (M_0 - \langle x_i(t),v \rangle)+ a_{iL}(t) \Big ( \frac{M_0+m_0}{2}- \langle x_i(t),v \rangle \Big) \\
& = \Lambda ( M_0 - \langle x_i(t),v \rangle )- a_{iL}(t)\frac{M_0-m_0}{2} \\
& \leq \Big( \Lambda M_0 -\frac{\Gamma}{N-1}\frac{M_0-m_0}{2} \Big) - \Lambda \langle x_i(t),v \rangle.
\end{split}
\end{equation}
Integrating over $[0, t]$ with $t \in [0,\sigma],$ we find 
$$ \langle x_i(t),v \rangle \leq e^{-\Lambda t} \langle x_i(0),v \rangle + \Big(M_0-\frac{1}{N-1}\frac{\Gamma}{\Lambda}\frac{M_0-m_0}{2}\Big) (1-e^{-\Lambda t}).$$
Taking $t=\sigma$ in the above inequality, we can find 
$$ \langle x_i(\sigma),v \rangle \leq M_0 - (1-e^{-\Lambda \sigma})\frac{1}{N-1} \frac{\Gamma}{\Lambda} \frac{M_0-m_0}{2}. $$
Denoting 
\begin{equation}\label{delta_minus1} \delta_{-}^1:=(1-e^{-\Lambda \sigma})\frac{1}{2(N-1)}\frac{\Gamma}{\Lambda} \Big(1-\frac{m_0}{M_0}\Big), \end{equation}
we have the inequality
\begin{equation}
\langle x_i(\sigma),v \rangle \leq (1-\delta_{-}^1)M_0, \ \forall i \in \{k+1,\dots,N\} \setminus \{L \}.
\end{equation}
Consider now $t \in [\sigma, 6\tau]$.
\begin{equation}\label{s31}
\begin{split}
\frac{d}{d t} \langle x_i(t),v \rangle & = \sum_{j \neq i} a_{ij}(t)\langle x_j(t)-x_i(t),v \rangle  \\
& \leq \Lambda(M_0-\langle x_i(t),v \rangle), \ \forall i \in \{k+1,\dots,N\} \setminus \{L\}.
\end{split}
\end{equation}
Integrating over $[\sigma, t]$ with $t \in [\sigma, 6\tau],$ we have
\begin{equation*}
\begin{split}
\langle x_i(t),v \rangle & \leq e^{-\Lambda(t-\sigma)} \langle x_i(\sigma),v \rangle+ M_0(1-e^{-\Lambda (t-\sigma)}) \\
& \leq e^{-\Lambda(t-\sigma)}(1-\delta_{-}^1)M_0 + M_0(1-e^{-\Lambda (t-\sigma)}) \\
& = M_0 \left (1-\delta_{-}^1 e^{-\Lambda (t-\sigma)}\right ) \\
& \leq M_0 \left (1-\delta_{-}^1 e^{-6\Lambda \tau}\right ), \ 
\end{split}
\end{equation*}
Then, 
\begin{equation}
\langle x_i(t),v \rangle \leq M_0 (1-\delta_{-}^1 e^{-6\Lambda \tau}), \ \forall i \in \{k+1,\dots,N\} \setminus \{L \}, \ \forall t \in [\sigma, 6\tau].
\end{equation}
Consider now $i \in \{ 1, \dots, k \}$ and fix $i_1 \in \{k+1,\dots, N\}\setminus \{L\}$. Taking $t \in [\sigma,6\tau],$ we have 
\begin{equation}
\begin{split}
\frac{d}{d t} \langle x_{i}(t),v \rangle & = \sum_{\substack{j \neq i \\ j \neq i_1}} a_{i j}(t) \langle x_j(t)-x_{i}(t),v \rangle + \sum_{j=1}^{h} \epsilon_{i j}(t) \langle y_j(t-\tau)-x_{i}(t),v \rangle \\
& + a_{i i_1}(t) \langle x_{i_1}(t)-x_{i}(t),v \rangle \\
& \leq \Big( \sum_{j \neq i} a_{ij}(t)-a_{ii_1}(t) \Big)  (M_0-\langle x_{i}(t),v \rangle)+ \frac{h}{N+h-1}\Lambda (M_0-\langle x_{i}(t),v \rangle) \\
& + a_{ii_1}(t)(M_0(1-\delta_{-}^1 e^{- 6\tau\Lambda})-\langle x_{i}(t),v \rangle) \\
& \leq \frac{N-1}{N+h-1} \Lambda (M_0-\langle x_{i}(t),v \rangle)+ \frac{h}{N+h-1}\Lambda (M_0-\langle x_{i}(t),v \rangle) \\
& - \frac{1}{N+h-1}\Gamma M_0 \delta_{-}^1 e^{- 6\tau\Lambda}\\
& = \Lambda M_0 \Big(1-\frac{1}{N+h-1}\frac{\Gamma}{\Lambda}\delta_{-}^1 e ^{- 6\tau\Lambda}\Big)-\Lambda \langle x_{i}(t),v \rangle.
\end{split}
\end{equation}
Integrating over $[\sigma, t]$ with $t \in [\sigma,6\tau]$ we find 
\begin{equation}
\begin{split}
\langle x_{i}(t),v \rangle & \leq e^{-\Lambda(t-\sigma)} \langle x_{i}(\sigma),v \rangle + M_0 \Big(1-\frac{1}{N+h-1}\frac{\Gamma}{\Lambda}\delta_{-}^1 e ^{- 6\tau\Lambda}\Big)(1-e^{-\Lambda(t-\sigma)}) \\
& \leq M_0 \Big(1-\frac{1}{N+h-1}\frac{\Gamma}{\Lambda}\delta_{-}^1 e ^{- 6\tau\Lambda}(1-e^{-\Lambda(t-\sigma)}) \Big).
\end{split}
\end{equation}
Shrinking to $t \in [2\tau,6\tau]$ we find that 
\begin{equation} \label{C}
\langle x_{i}(t),v \rangle \leq M_0 \Big(1-\frac{1}{N+h-1}\frac{\Gamma}{\Lambda}\delta_{-}^1 e ^{- 6\tau\Lambda}(1-e^{-\Lambda \tau}) \Big), \ \forall i \in \{1,\dots,k\}.
\end{equation}
Using the state $i_1 \in \{k+1, \dots, N\}\setminus \{L\},$ we can find an upper bound for $i_L$ too. Indeed, for $t \in [\sigma,6\tau],$ analogously to \eqref{s31}, we obtain
\begin{equation}
\begin{split}
& \frac{d}{d t} \langle x_L(t),v \rangle = \sum_{\substack{j \neq L \\ j \neq i_1}} a_{Lj} \langle x_j(t)-x_L(t),v \rangle + a_{Li_1} \langle x_{i_1}(t)-x_L(t),v \rangle \\
& \leq \Lambda M_0 \Big( 1-\frac{1}{N-1}\frac{\Gamma}{\Lambda}\delta_{-}^1e^{-6\Lambda\tau} \Big)-\Lambda \langle x_L(t),v \rangle.
\end{split}
\end{equation}
Integrating over $[\sigma, t]$ with $t \in [\sigma, 6\tau],$ we find
\begin{equation}
\langle x_L(t),v \rangle \leq e^{-\Lambda(t-\sigma)}\langle x_L(\sigma),v \rangle + M_0\Big(1-\frac{1}{N-1}\frac{\Gamma}{\Lambda}\delta_{-}^1e^{-6\Lambda\tau}\Big)(1-e^{-\Lambda(t-\sigma)}).
\end{equation}
Shrinking to $t \in [2\tau,6\tau],$ we have the estimate
\begin{equation} \label{Cbis}
\langle x_L(t),v \rangle \leq M_0 \Big( 1- \frac{1}{N-1}\frac{\Gamma}{\Lambda}\delta_{-}^1e^{-6\Lambda\tau}(1-e^{-\Lambda\tau})\Big).
\end{equation}
Consider now $i \in \{1,\dots,h\}$ and fix $i_2 \in \{1,\dots,k\}.$ Taking $t \in [3\tau,6\tau],$ we have 
\begin{equation}
\begin{split}
\frac{d}{d t} \langle y_{i}(t), &v \rangle  = \sum_{j \neq i} b_{i j}(t) \langle y_j(t)-y_{i}(t),v \rangle + \sum_{\substack{j=1 \\ j \neq i_2}}^{k} \eta_{ij}(t) \langle x_j(t-\tau)-y_{i}(t),v \rangle\\
&\hspace{1,2 cm} + \eta_{i i_2}(t)\langle x_{i_2}(t-\tau)-y_{i}(t),v \rangle \\
& \leq \frac{M-1}{M+k-1} \Lambda (M_0-\langle y_{i}(t),v \rangle)+ \Big(\sum_{j=1}^{k}\eta_{ij}(t)-\eta_{ii_2}(t)\Big)(M_0-\langle y_{i}(t),v \rangle)\\
&\hspace{1,2 cm} +\eta_{ii_2}(t) \Big[ M_0 \Big( 1-\frac{1}{N+h-1}\frac{\Gamma}{\Lambda}\delta_{-}^1e^{-6\tau\Lambda }(1-e^{-\Lambda \tau}) \Big)-\langle y_{i}(t),v \rangle \Big] \\
& \leq \Lambda M_0 \Big( 1-\frac{1}{(N+h-1)(M+k-1)}\Big(\frac{\Gamma}{\Lambda}\Big)^2\delta_{-}^1e^{- 6\tau\Lambda}(1-e^{-\Lambda \tau}) \Big) -\Lambda \langle y_{i}(t),v \rangle.
\end{split}
\end{equation}
Integrating over $[3\tau, t],$ with $t \in [3\tau,6\tau],$ we find that 
\begin{equation}
\begin{split}
\langle y_{i}(t),v \rangle & \leq e^{-\Lambda(t-3\tau)}\langle y_{i}(3\tau),v \rangle \\
& + M_0 \Big( 1- \frac{1}{(N+h-1)(M+k-1)}\Big(\frac{\Gamma}{\Lambda}\Big)^2\delta_{-}^1 e^{- 6\tau\Lambda}(1-e^{-\Lambda \tau}) \Big) (1-e^{-\Lambda(t-3\tau)}) \\
& \leq M_0 \Big( 1- \frac{1}{(N+h-1)(M+k-1)}\Big(\frac{\Gamma}{\Lambda}\Big)^2\delta_{-}^1 e^{- 6\tau\Lambda}(1-e^{-\Lambda \tau})(1-e^{-\Lambda(t-3\tau)}) \Big). 
\end{split}
\end{equation}
Shrinking to $t \in [4\tau,6\tau]$ we find that 
\begin{equation} \label{A}
\langle y_{i}(t),v \rangle \leq M_0 \Big( 1- \frac{1}{(N+h-1)(M+k-1)}\Big(\frac{\Gamma}{\Lambda}\Big)^2\delta_{-}^1 e^{- 6\tau\Lambda}(1-e^{-\Lambda \tau})^2 \Big), \ \forall i \in \{1,\dots,h\}.
\end{equation}
Finally, consider $i \in \{h+1,\dots,M\}$ and fix $j_1 \in \{1,\dots,h\}.$ Taking $t \in [4\tau,6\tau],$ with analogous computation, we find
\begin{equation}
\begin{split}
& \frac{d}{d t} \langle y_{i}(t),v \rangle = \sum_{\substack{j \neq i \\ j \neq j_1}} b_{i j}(t) \langle y_j(t)-y_{i}(t),v \rangle + b_{i j_1}(t) \langle y_{j_1}(t)-y_{i}(t),v \rangle \\
& \leq \Lambda M_0 \Big( 1-\frac{1}{(M-1)(N+h-1)(M+k-1)}\Big(\frac{\Gamma}{\Lambda}\Big)^3\delta_{-}^1e^{-6\tau\Lambda}(1-e^{-\Lambda\tau})^2 \Big)-\Lambda \langle y_{i}(t),v \rangle.
\end{split}
\end{equation}
Integrating over $ [4\tau, t]$ with $t \in [4\tau,6\tau],$ we have that
\begin{equation*}
\begin{split}
& \langle y_{i}(t),v \rangle \leq e^{-\Lambda(t-4\tau)}\langle y_{i}(4\tau),v\rangle  \\
& + M_0 \Big[ 1- \frac{1}{(M-1)(N+h-1)(M+k-1)}\Big(\frac{\Gamma}{\Lambda}\Big)^3\delta_{-}^1e^{-6\Lambda\tau}(1-e^{-\Lambda\tau})^2\Big] (1-e^{-\Lambda(t-4\tau)}) \\
& \leq M_0 \Big[ 1- \frac{1}{(M-1)(N+h-1)(M+k-1)}\Big(\frac{\Gamma}{\Lambda}\Big)^3\delta_{-}^1e^{-6\Lambda\tau}(1-e^{-\Lambda\tau})^2(1-e^{-\Lambda(t-4\tau)})\Big]. 
\end{split}
\end{equation*}
Shrinking to $t \in [5\tau,6\tau],$ we find that 
\begin{equation} \label{F}
\begin{split}
\langle y_{i}(t),v \rangle \leq M_0 \Big[ 1- \frac{1}{(M-1)(N+h-1)(M+k-1)}\Big(\frac{\Gamma}{\Lambda}\Big)^3 & \delta_{-}^1e^{-6\Lambda\tau}(1-e^{-\Lambda\tau})^3 \Big] , \\ 
& \forall i \in \{h+1,\dots,M\}. 	 
\end{split}
\end{equation}
Then, the inequality \eqref{F} holds for all the non-leader of the second population, for $t \in [5\tau,6\tau]$. Moreover, one can notice that the right-hand side of \eqref{F} is larger than the right-hand side of \eqref{A}. Thus, the estimate \eqref{F} holds for all the states of the second population. Since the right-hand side of \eqref{F} is larger than the one of \eqref{C} and \eqref{Cbis}, we have that \eqref{F} holds for all the states of the first population too, for $t \in [5\tau,6\tau]$. \\

\emph{Step 2}: Assume now that $L \in \{1,\dots,k\}$ is such that $\langle x_L(s),v \rangle = m_0$ for some $s \in [-\tau,0].$ By continuity, then there exists a closed interval $[\alpha_L,\beta_L] \subset [-\tau,0]$ such that 
$$ m_0 \leq \langle x_L(t),v \rangle \leq \frac{M_0+m_0}{2}, \ t \in [\alpha_L,\beta_L].$$
Eventually choosing a smaller $\sigma$ in \eqref{sigma}, we may assume $\beta_L-\alpha_L=\sigma.$
Consider $i \in \{1,\dots,h\}$ and $t \in [\alpha_L+\tau,\beta_L+\tau].$
Then, 
\begin{equation*}
\begin{split}
\frac{d}{d t} \langle y_{i}(t),v \rangle & = \sum_{j \neq i} b_{ij}(t) \langle y_j(t)-y_i(t),v \rangle + \sum_{\substack{j=1 \\ j \neq L}}^{k} \eta_{ij}(t) \langle x_j(t-\tau)-y_i(t),v \rangle \\
& + \eta_{iL}(t) \langle x_L(t-\tau)-y_i(t),v \rangle \\
& \leq \frac{M-1}{M+k-1}\Lambda(M_0-\langle y_i(t),v \rangle )
+\Big(\sum_{j=1}^{k}\eta_{ij}-\eta_{iL}\Big)(M_0-\langle y_i(t),v \rangle) \\
& +\eta_{iL}(t) \Big( \frac{M_0+m_0}{2}-\langle y_i(t),v\rangle \Big), 
\end{split}
\end{equation*}
and so,
\begin{equation}\label{55}
\begin{split}
\frac{d}{d t} \langle y_{i}(t),v \rangle& \leq \frac{M-1}{M+k-1}\Lambda(M_0-\langle y_i(t),v \rangle)+\frac{k}{M+k-1}(M_0-\langle y_i(t),v \rangle) \\
& -\eta_{iL}(t)(M_0-\langle y_i(t),v \rangle) +\eta_{iL}(t)\Big( \frac{M_0+m_0}{2}-\langle y_i(t),v\rangle \Big) \\
& = \Lambda (M_0-\langle y_i(t),v \rangle) -\eta_{iL}(t)\frac{M_0-m_0}{2} \\
& \leq \Big( \Lambda M_0 -\frac{1}{M+k-1}\Gamma \frac{M_0-m_0}{2}\Big) - \Lambda \langle y_i(t),v \rangle.
\end{split}
\end{equation}
Integrating \eqref{55} on $[\alpha_L+\tau, t]$ with $t \in [\alpha_L+\tau,\beta_L+\tau],$ we have that 
\begin{equation*}
\langle y_i(t),v \rangle \leq e^{-\Lambda(t-\alpha_L-\tau)}\langle y_i(\alpha_L+\tau),v\rangle + \Big(M_0 -\frac{1}{M+k-1}\frac{\Gamma}{\Lambda}\frac{M_0-m_0}{2}\Big)(1-e^{-\Lambda(t-\alpha_L-\tau)}).
\end{equation*}
Putting $t=\beta_L+\tau$ in the equation above, we find 
\begin{equation}\label{56}
\langle y_i(\beta_L+\tau),v \rangle \leq M_0-\frac{1}{M+k-1}\frac{\Gamma}{\Lambda}\frac{M_0-m_0}{2}(1-e^{-\Lambda\sigma}).
\end{equation}
From \eqref{56}, denoting,
\begin{equation} \label{delta_minu2}
\delta_{-}^2:=\frac{1}{2(M+k-1)}\frac{\Gamma}{\Lambda}\Big(1-\frac{m_0}{M_0}\Big)(1-e^{-\Lambda\sigma}), 
\end{equation} 
we deduce that 
\begin{equation}
\langle y_i(\beta_L+\tau),v \rangle \leq (1-\delta_{-}^2)M_0, \ \forall i \in \{1,\dots, h\}.
\end{equation}
Consider now $t \in [\beta_L+\tau,6\tau].$ Then,
\begin{equation}
\begin{split}
\frac{d}{d t} \langle y_i(t),v \rangle & = \sum_{j \neq i} b_{ij}(t) \langle y_j(t)-y_i(t),v \rangle + \sum_{j=1}^{k} \eta_{ij}(t)\langle x_j(t-\tau)-y_i(t),v \rangle \\
& \leq \Lambda(M_0-\langle y_i(t),v \rangle).
\end{split}
\end{equation}
Integrating on $[\beta_L+\tau, t]$ with  $t \in [\beta_L+\tau,6\tau],$ we have that 
\begin{equation}
\begin{split}
\langle y_i(t),v \rangle & \leq e^{-\Lambda(t-\beta_L-\tau)}\langle y_i(\beta_L+\tau),v \rangle+ M_0(1-e^{-\Lambda(t-\beta_L-\tau)})\\
& \leq e^{-\Lambda(t-\beta_L-\tau)}(1-\delta_{-}^2)M_0+ M_0(1-e^{-\Lambda(t-\beta_L-\tau)})\\
& = M_0(1-\delta_{-}^2e^{-\Lambda(t-\beta_L-\tau)}) \\
& \leq M_0(1-\delta_{-}^2e^{-6\tau\Lambda}),
\end{split}
\end{equation}
where the last inequality is obtained observing that $t-\beta_L-\tau \leq 6\tau.$ Then, 
\begin{equation}
\langle y_i(t),v \rangle \leq M_0(1-\delta_{-}^2e^{-6\tau\Lambda}), \ \forall i \in \{1,\dots,h\}, \ t \in [\beta_L+\tau,6\tau], \ i \in \{1,\dots,h\}.
\end{equation}
Consider now $i \in \{h+1,\dots,M\}$ and fix $i_1 \in \{1,\dots,h\}.$ Taking $t \in [\beta_L+\tau,6\tau],$ we have 
\begin{equation}\label{57}
\begin{split}
\frac{d}{d t} \langle y_{i}(t), &v \rangle  = \sum_{\substack{j \neq i \\ j \neq i_1}} b_{ij}(t) \langle y_j(t)-y_{i}(t),v \rangle + b_{ii_1}(t) \langle y_{i_1}(t)-y_{i}(t),v \rangle \\
& \leq \Big(\sum_{j \neq i}b_{ij}(t)-b_{ii_1}(t)\Big)(M_0-\langle y_{i}(t),v \rangle) + b_{ii_1}(t)[M_0(1-\delta_{-}^2e^{-6\tau\Lambda})-\langle y_{i}(t),v \rangle] \\
& \leq \Lambda M_0 \Big[ 1-\frac{1}{M-1}\frac{\Gamma}{\Lambda}\delta_{-}^2e^{-6\tau\Lambda}\Big]-\Lambda \langle y_{i}(t),v \rangle. 
\end{split}
\end{equation}
Applying the Grönwall inequality  on $[\beta_L+\tau, t]$ with $t \in [\beta_L+\tau,6\tau],$ from \eqref{57} we find that
\begin{equation*}
\begin{split}
\langle y_{i}(t),v \rangle & \leq e^{-\Lambda(t-\beta_L-\tau)} \langle y_{i}(\beta_L+\tau),v \rangle + M_0\Big[ 1-\frac{1}{M-1}\frac{\Gamma}{\Lambda}\delta_{-}^2e^{-6\tau\Lambda} \Big](1-e^{-\Lambda(t-\beta_L-\tau)})\\
& \leq M_0\Big[ 1-\frac{1}{M-1}\frac{\Gamma}{\Lambda}\delta_{-}^2e^{-6\tau\Lambda} (1-e^{-\Lambda(t-\beta_L-\tau)})\Big].
\end{split}
\end{equation*}
Shrinking to $t \in [2\tau,6\tau],$ noticing that $t-\beta_L-\tau \geq \tau,$ we have that 
\begin{equation}
\langle y_{i}(t),v \rangle \leq M_0 \Big[1-\frac{1}{M-1}\frac{\Gamma}{\Lambda}\delta_{-}^2e^{-6\tau\Lambda}(1-e^{-\Lambda \tau})\Big], \ \forall i \in \{h+1,\dots,M\}.
\end{equation}
Consider now $i \in \{1,\dots,k\}$ and fix $i_2 \in \{1,\dots,h\}.$ Notice that could be that $i_2=i_1.$ For $t \in [2\tau, 6\tau],$ we have
\begin{equation*}
\begin{split}
\frac{d}{d t} \langle x_{i}(t),v \rangle&  = \sum_{j \neq i} a_{ij}(t) \langle x_j(t)-x_{i}(t),v \rangle + \sum_{\substack{j=1 \\ j \neq i_2}}^{h} \epsilon_{ij}(t) \langle y_j(t-\tau)-x_{i}(t),v \rangle \\
& + \epsilon_{ii_2}(t) \langle y_{i_2}(t)-x_i(t),v \rangle \\
& \leq  \Lambda M_0 \Big[1-\frac{1}{N+h-1}\frac{\Gamma}{\Lambda}\delta_{-}^2e^{-6\tau\Lambda}\Big]-\Lambda \langle x_{i}(t),v \rangle.
\end{split}
\end{equation*} 
Integrating over $[2\tau, t]$ with $t \in [2\tau,6\tau],$ we find
\begin{equation*}
\langle x_{i}(t),v \rangle \leq e^{-\Lambda(t-2\tau)}\langle x_{i}(2\tau),v \rangle + M_0 \Big(1-\frac{1}{N+h-1}\frac{\Gamma}{\Lambda}\delta_{-}^2e^{-6\tau\Lambda}\Big)(1-e^{-\Lambda(t-2\tau)}).
\end{equation*}
Shrinking to $t \in [3\tau,6\tau]$ we have 
\begin{equation} \label{A1}
\langle x_{i}(t),v \rangle \leq M_0 \Big( 1-\frac{1}{N+h-1}\frac{\Gamma}{\Lambda}\delta_{-}^2e^{-6\tau\Lambda}(1-e^{-\Lambda\tau})\Big), \ \forall i \in \{1,\dots,k\}.
\end{equation}
Finally, we consider $i \in \{k+1,\dots,N\}$ and fix $i_3 \in \{1,\dots,k\}.$ Taking $t \in [3\tau,6\tau]$ we have
\begin{equation*}
\begin{split}
\frac{d}{d t} \langle x_i(t),v \rangle& = \sum_{\substack{j \neq i \\ j \neq i_3}} a_{ij}(t) \langle x_j(t)-x_i(t),v \rangle + a_{ii_3}(t)\langle x_{i_3}(t)-x_i(t),v \rangle \\
& \leq  \Lambda M_0 \Big[ 1-\frac{1}{(N-1)(N+h-1)}\Big(\frac{\Gamma}{\Lambda}\Big)^2\delta_{-}^2e^{-6\tau\Lambda}(1-e^{-\Lambda\tau})\Big]-\Lambda \langle x_i(t),v \rangle. 
\end{split}
\end{equation*} 
Integrating the inequality above over $[3\tau, t]$ for $t \in [3\tau,6\tau],$ we find that
\begin{equation*}
\begin{split}
\langle x_i(t),v \rangle & \leq e^{-\Lambda(t-3\tau)} \langle x_i(3\tau),v \rangle \\
& + M_0 \Big[ 1-\frac{1}{(N-1)(N+h-1)}\Big(\frac{\Gamma}{\Lambda}\Big)^2\delta_{-}^2e^{-6\tau\Lambda}(1-e^{-\Lambda\tau})\Big](1-e^{-\Lambda(t-3\tau)}). 
\end{split}
\end{equation*}
Shrinking to $t \in [4\tau,6\tau],$ we finally have that 
\begin{equation}\label{F1}
\langle x_i(t),v \rangle \leq M_0 \Big( 1- \frac{1}{(N-1)(N+h-1)}\Big(\frac{\Gamma}{\Lambda}\Big)^2\delta_{-}^2e^{-6\tau\Lambda}(1-e^{-\Lambda\tau})^2\Big), \ \forall i \in \{k+1,\dots,N\}.
\end{equation} 
As in the previous case, the estimate \eqref{F1} holds for all the possible states of the system for $t \in [5\tau,6\tau].$ 
\\

\emph{Step 3}:
From {\em Step 1} and {\em Step 2}, using the definitions  \eqref{delta_minus1} and \eqref{delta_minu2}, we have then
\begin{equation} \label{N150}
\begin{split}
& \langle x_i(t),v \rangle \leq 
 M_0 \Big[ 1- \frac{1}{2(N-1)(M-1)(M+k-1)(N+h-1)}\Big(\frac{\Gamma}{\Lambda}\Big)^4\times\\ &\hspace{6 cm}\times e^{-6\tau\Lambda}(1-e^{-\Lambda \sigma})(1-e^{-\Lambda\tau})^3\Big(1-\frac{m_0}{M_0}\Big)\Big], 
\end{split}
\end{equation}
$ \forall i \in \{1,\dots,N\},$ $t \in [5\tau,6\tau],$
and 
\begin{equation} \label{N250}
\begin{split}
& \langle y_i(t),v \rangle \leq M_0 \Big[ 1- \frac{1}{2(N-1)(M-1)(M+k-1)(N+h-1)}\Big(\frac{\Gamma}{\Lambda}\Big)^4\times\\
&\hspace{6 cm}\times e^{-6\tau\Lambda}(1-e^{-\Lambda \sigma})(1-e^{-\Lambda\tau})^3\Big(1-\frac{m_0}{M_0}\Big)\Big], 
\end{split}
\end{equation}
$\forall i \in \{1,\dots,M\},$ $t \in [5\tau,6\tau].$
Analogous estimates can be obtained if $m_0$ or $M_0$ are attained by scalar products 
$\langle y_i, v\rangle, $ $i=1, \dots, M.$ 
Since 
$$ (N-1)(M-1)(N+h-1)(M+k-1) \leq 4N^4, $$
from \eqref{N150} and \eqref{N250}, we obtain the second inequalities of \eqref{x} and \eqref{y}, respectively, with

\begin{equation} \label{gamma1}
\Gamma_1:=\frac{1}{8N^4}\Big(\frac{\Gamma}{\Lambda}\Big)^4 e^{-6\tau\Lambda}(1-e^{-\Lambda\tau})^3(1-e^{-\Lambda\sigma}).
\end{equation}

\emph{Step 4}:
Now, we focus on the lower bound in \eqref{x} and \eqref{y}. \\
Assume that there exists $R \in \{ k+1,\dots,N\}$ such that $\langle x_R(0),v \rangle = M_0.$ Then, as before, we have  that 
$$ \frac{M_0+m_0}{2} \leq \langle x_R(t),v \rangle \leq M_0, \ t \in [0,\sigma],$$
with $\sigma$ as in \eqref{sigma}.
Using similar arguments to the ones in  {\em Step 1}, we find that,  for $t \in [5\tau,6\tau],$ 
\begin{equation}\label{B0}
\begin{split}
\langle x_{i}(t),v \rangle \geq m_0 \Big[ 1+\frac{1}{(M-1)(N+h-1)(M+k-1)}\Big( \frac{\Gamma}{\Lambda}\Big)^3 \delta_{+}^1e^{-6\tau\Lambda}& (1-e^{-\Lambda\tau})^3 \Big], \\ 
& \forall i \in \{1,\dots,N\},
\end{split}
\end{equation}
and
\begin{equation}\label{B1}
\begin{split}
\langle y_{i}(t),v \rangle \geq m_0 \Big[ 1+\frac{1}{(M-1)(N+h-1)(M+k-1)}\Big( \frac{\Gamma}{\Lambda}\Big)^3 \delta_{+}^1e^{-6\tau\Lambda} & (1-e^{-\Lambda\tau})^3 \Big], \\
& \forall i \in \{1,\dots,M\},
\end{split}
\end{equation}
with 
\begin{equation} \label{delta_plus1}
\delta_{+}^1:=\frac{1}{2(N-1)}\frac{\Gamma}{\Lambda}(1-e^{-\Lambda\sigma})\Big(\frac{M_0}{m_0}-1 \Big).
\end{equation}

Suppose, instead,  that $\langle x_R(t),v \rangle = M_0$ for some $R \in \{1,\dots,k\}$ and for some $t \in [-\tau,0].$ Then, by continuity, there exists a closed interval $[\alpha_R,\beta_R]\subset[-\tau,0]$  such that 
$$ \frac{M_0+m_0}{2} \leq \langle x_R(t),v \rangle \leq M_0, \ t \in [\alpha_R,\beta_R].$$
Eventually choosing a smaller $\sigma$ in \eqref{sigma} above, we may assume that
$\beta_R-\alpha_R=\sigma.$ 
Arguing analogously to {\em Step 2}, we can obtain, for $t \in [5\tau,6\tau],$   
\begin{equation} \label{B}
\langle x_{i}(t),v \rangle \geq m_0 \Big[ 1+\frac{1}{(N-1)(N+h-1)}\Big( \frac{\Gamma}{\Lambda}\Big)^2\delta_{+}^2 e^{- 6\tau\Lambda}  (1-e^{-\Lambda \tau})^2 \Big], 
\quad \forall i \in \{1,\dots,N\},
\end{equation}
and 
\begin{equation} \label{B2}
\langle y_{i}(t),v \rangle \geq m_0 \Big[ 1+\frac{1}{(N-1)(N+h-1)}\Big( \frac{\Gamma}{\Lambda}\Big)^2\delta_{+}^2 e^{- 6\tau\Lambda}
 (1-e^{-\Lambda \tau})^2 \Big], 
\quad \forall i \in \{1,\dots,M\},
\end{equation}
with 
\begin{equation} \label{delta_plus2}
 \delta_{+}^2:= \frac{1}{2(M+k-1)}\frac{\Gamma}{\Lambda}(1-e^{-\Lambda\sigma})\Big(\frac{M_0}{m_0}-1\Big).
 \end{equation}
Now, note that the right-hand side of \eqref{B0} and \eqref{B1} are smaller than the right-hand side of \eqref{B} and \eqref{B2} and so, using the definitions \eqref{delta_plus1} and \eqref{delta_plus2}, we have that, for $t \in [5\tau,6\tau],$
\begin{equation*} 
\begin{split}
&\langle x_i(t),v \rangle \geq m_0 \Big[ 1+\frac{1}{2(N-1)(M-1)(M+k-1)(N+h-1)}\Big(\frac{\Gamma}{\Lambda}\Big)^4 \times \\
& \hspace{6 cm} \times e^{-6\tau\Lambda}(1-e^{-\Lambda \sigma})(1-e^{-\Lambda\tau})^3\Big(\frac{M_0}{m_0}-1\Big)\Big], 
\end{split}
\end{equation*}
$\forall i \in \{1,\dots,N\},$ 
and 
\begin{equation*} 
\begin{split}
& \langle y_i(t),v \rangle \geq m_0 \Big[ 1+\frac{1}{2(N-1)(M-1)(M+k-1)(N+h-1)}\Big(\frac{\Gamma}{\Lambda}\Big)^4 \times \\
& \hspace{6 cm} \times e^{-6\tau\Lambda}(1-e^{-\Lambda \sigma})(1-e^{-\Lambda\tau})^3\Big(\frac{M_0}{m_0}-1\Big)\Big], 
\end{split}
\end{equation*}
$ \forall i \in \{1,\dots, M\}.$  Using the definition  \eqref{gamma1}, from the last two inequalities we obtain the lower bounds in the lemma's statement. This completes the proof.
\end{proof}

\section{Asymptotic consensus}\label{cons}
In this section, we will show the asymptotic convergence to consensus of solutions to  \eqref{eq1}.

\begin{Definition}
For fixed  $v \in \mathbb{R}^d$ unit vector, for all $n \in \mathbb{N},$ one can define the quantities $M_n$ and $m_n$ as follows:
\begin{equation} \label{mn}
\begin{split}
m_n:= \min \Big \{ & \min_{i=1,\dots,k} \min_{t \in I_n} \langle x_i(t),v \rangle, \min_{i=k+1,\dots,N} \langle x_i(6n\tau),v \rangle, \\
&\hspace{3 cm} \min_{i=1,\dots,h} \min_{t \in I_n} \langle y_i(t),v \rangle, \min_{i=h+1,\dots,M} \langle y_i(6n\tau),v \rangle \Big \},
\end{split}
\end{equation}
\begin{equation} \label{Mn}
\begin{split}
M_n:= \max \Big \{ & \max_{i=1,\dots,k} \max_{t \in I_n} \langle x_i(t),v \rangle, \max_{i=k+1,\dots,N} \langle x_i(6n\tau),v \rangle, \\
& \hspace{3 cm} \max_{i=1,\dots,h} \max_{t \in I_n} \langle y_i(t),v \rangle, \max_{i=h+1,\dots,M} \langle y_i(6n\tau),v \rangle \Big \},
\end{split}
\end{equation}
with $I_n=[(6n-1)\tau,6n\tau].$ Notice that for $n=0$ we recover \eqref{m0} and \eqref{M0}.
\end{Definition}

\begin{Theorem}\label{main} Let $(x_i(t),y_j(t)),\ i=1,\dots,N, \ j=1,\dots,M,$ be a global classical solution to the system \eqref{eq1} with continuous initial conditions \eqref{init}.Then,  $(x_i(t),y_j(t)),\ i=1,\dots,N, \ j=1,\dots,M,$  achieve an asymptotic consensus in the sense of Definition \ref{def1}.
\end{Theorem}

\begin{proof}
Using \eqref{mn} and \eqref{Mn}, we define the quantities $D_n:=M_n-m_n$. 
Moreover, let us denote
$$ \Gamma_{1n}:= \frac{1}{8N^4}\Big(\frac{\Gamma}{\Lambda}\Big)^4 e^{-6\tau\Lambda}(1-e^{-\Lambda\tau})^3(1-e^{-\Lambda\sigma_n}),$$
where $\sigma_n:= \min \{ \tau, \frac{M_n-m_n}{4\Lambda M_0} \},$ for $n\ge 1,$ and $\sigma_0=\sigma$ as in \eqref{sigma}.   Then, $\Gamma_1 = \Gamma_{10}$ and $\Gamma_{1n} \in (0,1)$ if $M_n>m_n.$ Now, we use Lemma \ref{case1} with $t \in I_n, \ n \in \mathbb{N}$. For $n=1,$ we have
\begin{equation*}
D_1=M_1-m_1 \leq M_0 - \frac{\Gamma_{10}}{2}(M_0-m_0)-m_0-\frac{\Gamma_{10}}{2}(M_0-m_0)=(M_0-m_0)(1-\Gamma_{10})=D_0(1-\Gamma_{10}).
\end{equation*}
So, we find that $D_1 \leq (1-\Gamma_{10})D_0.$ Iterating the process of Lemma \ref{case1} we can find 
$$ D_{n+1}\leq (1-\Gamma_{1n})D_n, \ \forall n \in \mathbb{N}.$$
Let us denote
$$ \tilde{\sigma}(D):=\min \Big\{ \tau, \frac{D}{4\Lambda M_0} \Big\},$$
so that $\sigma_n = \tilde{\sigma}(D_n), \ \forall n \in \mathbb{N},$ $n\ge 1.$
Moreover, let be 
$$ \tilde{\Gamma}_{1}(D):= \frac{1}{8N^4}\Big(\frac{\Gamma}{\Lambda}\Big)^4 e^{-6\tau\Lambda}(1-e^{-\Lambda\tau})^3(1-e^{-\Lambda\tilde{\sigma}(D)}),$$
so that $\Gamma_{1n} = \tilde{\Gamma}_{1}(D_n),$ $\forall n \in \mathbb{N}.$  Therefore, we have
$$ D_{n+1}\leq (1-\tilde{\Gamma}(D_n))D_n.$$
Then, $\{ D_n \}_{n \in \mathbb{N}}$ is a non-negative and decreasing sequence. Let us calling $D$ the limit of $\{ D_n \}_{n \in \mathbb{N}}$ and, passing to the limit as $n$ goes to $+\infty$ in the above estimate, we find 
$$ D \leq (1-\tilde{\Gamma}_1(D))D,$$
that is true if and only if $\tilde{\Gamma}_1(D) \leq 0.$ This gives $D=0$ and, noticing that, 
$$ \langle x_i(t)-x_j(t),v \rangle \leq M_n-m_n = D_n$$ for all $i,j=1,\dots,N$, we have that 
$ \langle x_i(t)-x_j(t),v \rangle  \rightarrow 0 $
as $t \rightarrow + \infty$ and for all $i,j=1,\dots,N.$ The same holds for $\langle y_i(t)-y_j(t),v \rangle, $ with $i,j=1,\dots,M,$ and for $\langle x_i(t)-y_j(t),v \rangle, $ with $i=1,\dots,N$ and $j=1,\dots,M.$ \\
Notice that the result above can be obtained for each unit vector $v \in \mathbb{R}^d$. In particular, by considering  the canonical basis of $\mathbb{R}^d,$  $\{ e_h\}_{h=1}^d,$ and taking $v= e_h$ we have that
 $$|\langle x_i(t)-x_j(t), e_h \rangle| \rightarrow 0, $$
as $t \rightarrow + \infty,$ for all $i,j=1,\dots,N$ and $h=1, \dots, d.$ The same happens to 
$\vert \langle y_i(t)-y_j(t), e_h \rangle|,$ for all $i,j=1,\dots,M,$  and to  $\vert\langle x_i(t)-y_j(t), e_h \rangle|,$ for all $i=1,\dots,N$ and $j=1,\dots, M.$
Then, the system achieves asymptotic consensus. 
\end{proof}

\section{Numerical simulations and application to ecology}\label{num}

In this section, we present some numerical tests for the system \eqref{eq1} in the one-dimensional case, i.e. $d=1.$ We consider the weight functions $a_{ij}(t)$ and $b_{ij}(t)$ defined by 
$$ \psi(r,r')= \psi^*(r,r'):=\tilde{\psi}(|r-r'|), \quad r,r'\in [0, +\infty).$$
Meanwhile, the weight functions $\epsilon_{ij}(t)$ and $\eta_{ij}(t)$ are assumed to be constant.

In particular, we consider the functions
\begin{equation}
\begin{split}
& \tilde{\psi}(r):= e^{-(r-1)^2}, \ r \in [0,+\infty), \\
& \epsilon_{ij}(t):= \frac{K_1}{N+h-1}, \ \forall \ i \in \{1,\dots,k\}, \ j \in \{1,\dots,h\} \\
& \eta_{ij}(t):=\frac{K_2}{M+k-1}, \ \forall \ i \in \{1,\dots,h\},\ j \in \{1,\dots,k\},
\end{split}
\end{equation}
with $K_1, K_2$ positive constants.

\begin{figure}
\centering
\includegraphics[scale=0.55]{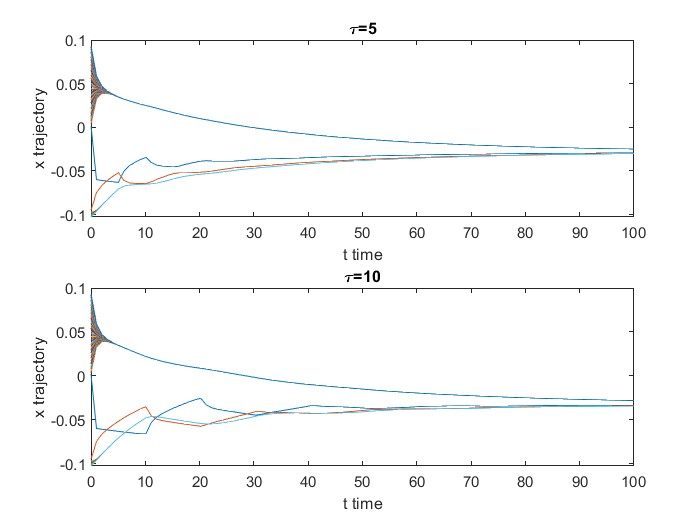}
\includegraphics[scale=0.55]{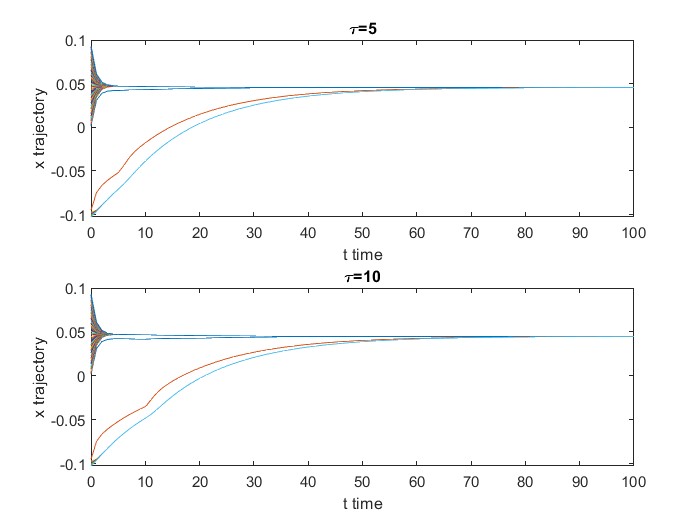}
\caption{Time evolution of solutions with different time delays, number of agents $N=50, \ M=5$, number of leaders $k=h=1.$}
\end{figure}

\begin{figure}
\centering
\includegraphics[scale=0.55]{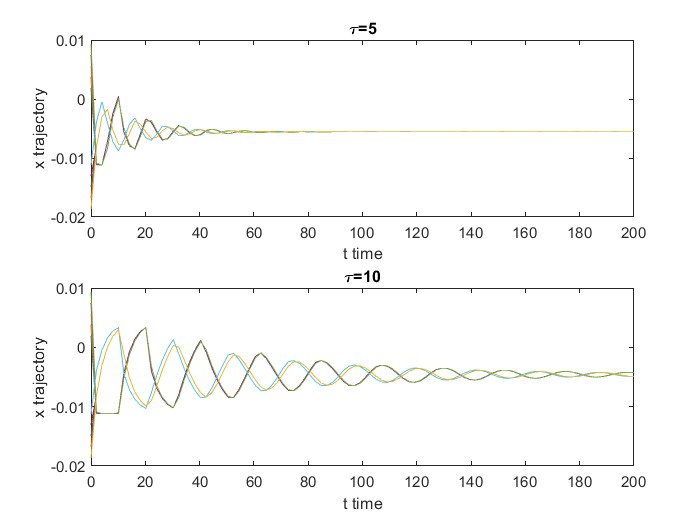}
\includegraphics[scale=0.55]{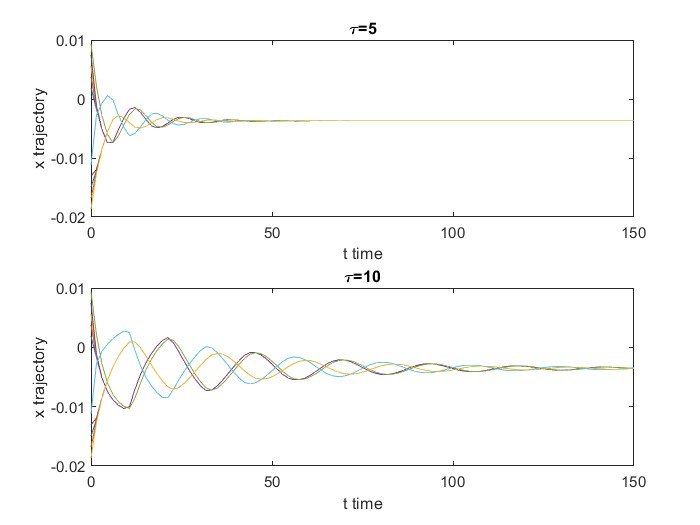}
\caption{Time evolution of solutions with different time delays, number of agents $N=M=5$, number of leaders $k=4,\ h=1.$}
\end{figure}

In Figure 1, the top two graphics illustrate a scenario where one population is larger than the other, yet the influence of the leaders from the smaller population overpowers that of the larger one ($K_1 \gg K_2$). As expected, the larger population tends to converge towards the consensus of the smaller population. The bottom two graphics depict a similar scenario but with equal influence from both sets of leaders ($K_1=K_2$). In this case, it is observed that the larger population pulls the smaller one towards its consensus.

Moving to Figure 2, the top two graphics illustrate a scenario where the total number of agents in both populations is equal, but the distribution of leaders differs ($k=4, \ h=1$). Furthermore, the solitary leader in the second population holds more influence than the others ($K_2\gg K_1$). It is noticeable that the system tends to a  consensus closer to the initial average of the population with only one leader. This is due to the different normalization factors of the weight functions.
In the bottom two graphics of Figure 2, a similar scenario is presented, but with equally strong influence from leaders on both sides ($K_1=K_2$). Here, it is observed that the consensus converges towards a mean value of the initial states. \\

In this paper, we used a Hegselmann-Krause type model to explore social dynamics and opinion formation in a set of two interacting populations, particularly in the context of discussing ecology strategies and sustainable development. We conducted simulations to investigate two different scenarios:
\begin{enumerate}

\item Equal Influence Scenario (Figure 3, Left): In this scenario, all agents are considered equal, meaning the total population coincides with the leaders' subgroup. However, only one group has a mild influence on the other group  ($k=h=N=M$, $K_1=0.3$, $K_2=0$). This could represent a situation with only one group spreading ecological information on social media.

\item Asymmetric Influence Scenario (Figure 3, Right): In this scenario, one group has a significant influence over a subgroup of the other population ($K_1=30$, $K_2=0$, $k<h$). This could represent a scenario where a community of scientists interacts with a leading group in the other population, e.g. a politicians' group.

\end{enumerate}

We observed that with a fixed time delay ($\tau=5$), consensus is reached more rapidly in the second scenario. This implies that exerting a strong influence on decision-makers who in turn influence the entire population leads to a faster consensus formation.
Therefore, we conclude that the most effective strategy for raising awareness on ecological topics is to exert a strong influence on key decision-makers who can influence the entire population. This highlights the importance of targeting influential individuals or groups in shaping public opinion and fostering consensus on issues, e.g., related to ecology and sustainability.
\begin{figure}
\centering
\includegraphics[scale=0.3]{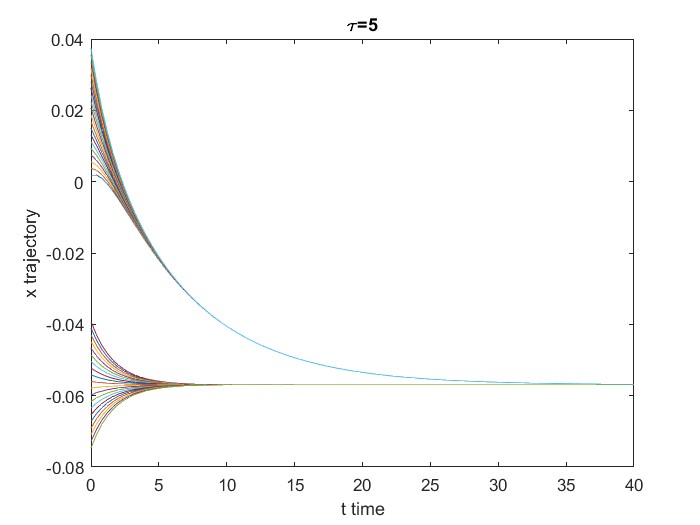}
\includegraphics[scale=0.3]{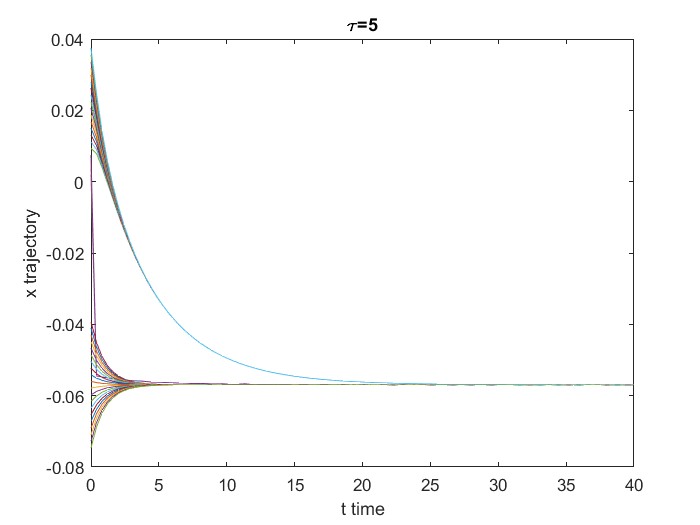}
\caption{Opinion formation in an ecological discussion: to the left, number of agents $N=M=20$, number of leaders $k=h=20;$ to the right, number of agents $N=M=20$, number of leaders $k=4, \ h=20.$ }
\end{figure}

\bigskip
\noindent {\bf Acknowledgements.} The authors are members of  {\it Gruppo Nazionale per l'Analisi Ma\-te\-matica, la Probabilit\`a e le loro Applicazioni (GNAMPA)} of the Istituto Nazionale di Alta Matematica (INdAM). They are also  members of {\it UMI ``CliMath"}.

C. Pignotti  is partially supported by PRIN 2022  (2022238YY5) {\it Optimal control problems: analysis,
approximation and applications}, PRIN-PNRR 2022 (P20225SP98) {\it Some mathematical approaches to climate change and its impacts}, and by INdAM GNAMPA Project {\it ``Modelli alle derivate parziali per interazioni multiagente non 
simmetriche"}(CUP E53C23001670001).

\end{document}